\documentclass[11pt, preprint, 1p, number, times]{elsarticle}
\usepackage{amsmath,amssymb, amsthm, amsfonts, color}
\numberwithin{equation}{section}
\newtheorem{theorem}{Theorem}[section]
\newtheorem{lemma}[theorem]{Lemma}

\newtheorem{corollary}[theorem]{Corollary}

\theoremstyle{remark}
\newtheorem{remark}[theorem]{Remark}
\newcommand{\be}{\begin{equation}}
\newcommand{\ee}{\end{equation}}
\newcommand{\beq}{\begin{equation}}
\newcommand{\enq}{\end{equation}}

\newcommand{\R}{{\mathbb{R}}}
\newcommand{\Nn}{{\mathbb{N}}}
\newcommand{\Cc}{{\mathbb{C}}}
\newcommand{\C}{{\mathbb{C}}}

\def\({\left(}
\def\){\right)}

\newcommand{\Zz}{{\mathbb{Z}}}
\newcommand{\supp}{{\mbox{\rm supp}}}

\newcommand{\von}{\varepsilon}
\newcommand{\eps}{\varepsilon}
\renewcommand{\epsilon}{\varepsilon}

\title{On a class of non-self-adjoint periodic eigenproblems with boundary and interior singularities\tnoteref{t1}\tnoteref{t2}}
\author[lb]{Lyonell Boulton\corref{cor1}}\ead{L.Boulton@ma.hw.ac.uk}
\author[ml]{Michael Levitin\fnref{grant}}\ead{levitin@reading.ac.uk}
\author[mm]{Marco Marletta}\ead{Marco.Marletta@cs.cardiff.ac.uk}

\tnotetext[t2]{We are grateful to Brian Davies for attracting our 
interest to the problem, and for numerous useful discussions. We also thank John Weir for his comments on an earlier version of this manuscript.}
\tnotetext[t1]{November 2008; revised August 2009}
\cortext[cor1]{Corresponding author}
\fntext[grant]{The research of this author was partially supported by the EPSRC grant EP/D054621.}
\address[lb]{Department of Mathematics, Heriot-Watt University,
and
Maxwell Institute for Mathematical Sciences,
Riccarton, Edinburgh EH14 4AS, United Kingdom}
\address[ml]{Department of Mathematics and Statistics, Reading University, Whiteknights, PO Box 220, Reading RG6 6AX, United Kingdom}
\address[mm]{Cardiff School of Mathematics, Cardiff University,
and
Wales Institute of Mathematical and Computational Sciences, Senghennydd Road, Cardiff CF24 4AG, United Kingdom}

\begin{document}

\begin{abstract}
We prove that all the eigenvalues of a certain highly non-self-adjoint 
Sturm-Liouville differential operator are real. The results presented
are motivated and extend those recently found by various authors 
(\cite{Benilov}, \cite{Davies} and \cite{Weir2}) 
on the stability of a model describing small oscillations 
of a thin layer of fluid inside a rotating cylinder.
\end{abstract}


\maketitle

\section{The general problem class}\label{section:1}
We consider on the interval $(-\pi,\pi)$ the singular 
non-symmetric differential
equation
\be Lu:=i\von \frac{d}{dx}\left(f(x)\frac{du}{dx}\right) + i \frac{du}{dx} = 
\lambda u, 
\label{eq:1}
\ee
in which $f$ is a $2\pi$-periodic function having the following properties:
\be\label{eq:f_sym}
f(x+\pi) = -f(x)\,,\qquad f(-x) = -f(x)\,; 
\ee
and also 
\be\label{eq:f_pos}
f(x)>0\qquad\text{for }x\in (0,\pi)\,. 
\ee
In particular
it follows that $f(\pi\Zz)=0$. We assume that $f$ is continuous, and differentiable 
except possibly at a finite number of points, the points of non-differentiability 
excluding $\pi\Zz$. We assume that $f'(0) = 2/\pi$ and that $0<\von<\pi$. 

Our interest in \eqref{eq:1} is primarily motivated by \cite{Benilov}
and \cite{Davies}, where
$f(x)=(2/\pi)\sin  x$, and  therefore \eqref{eq:1} takes the form
\be i\tilde\von \frac{d}{dx}\left(\sin x\frac{du}{dx}\right) + i \frac{du}{dx} = 
\lambda u, 
\label{eq:1sin}
\ee
with $0<\tilde\von<2$. Eq. \eqref{eq:1sin} arises in fluid dynamics, and describes small oscillations of a thin layer of fluid inside a rotating cylinder. It has been long conjectured that, despite the fact that \eqref{eq:1sin} is highly non-self-adjoint, all the eigenvalues of \eqref{eq:1sin} are real. 
For a fixed non-zero $\tilde\eps$ the conjecture has been supported by some numerical evidence, see e.g. \cite{ChuPel}, who have also proved that eigenvalues accumulate at infinity along the real line. Davies \cite{Davies} 
established that {  the spectrum is discrete, nonetheless the eigenfunctions do not form a basis. He also obtained a number of very interesting estimates. Recently Weir \cite{Weir,Weir2} proved, using Davies' estimates, that the spectrum of the highly non-selfadjoint problem \eqref{eq:1sin} is indeed purely real and just very recently Davies and Weir \cite{Davies-Weir} studied the asymptotics of the eigenvalues as $\tilde{\epsilon}\to 0$.}

The goal of our paper is to extend Weir's result and to prove that all the eigenvalues 
of \eqref{eq:1} are  real under minimal restrictions on the coefficient $f(x)$. 
In order not to hide the fundamental simplicity of our methods under
a burden of technical details, the proof that there is no essential
spectrum -- in fact, the resolvent operator has bounded kernel -- will
be given in an article currently in preparation.
One fundamental condition is the 
\emph{anti-symmetry} of $f$ with respect to $x$, see \eqref{eq:f_sym}, and also 
some symmetry properties of the whole operator, see below. This makes \eqref{eq:1}
reminiscent  (but not identical) to a wide class of so-called PT-symmetric equations: 
non-self-adjoint problems, which are not similar to self-adjoint, but which 
nevertheless possess purely real spectra due to some obvious and hidden symmetries. 
For examples, surveys, and interesting recent developments see e.g. \cite{BenBoe, DorDunTat1, DorDunTat2, Shin1, Shin2, LanTre, Kr} and references therein.

An additional difficulty in dealing with \eqref{eq:1} (compared to standard PT-symmetric problems) is the presence of 
singularities at $x=-\pi,0,\pi$. On the other hand, it is exactly these singularities 
which allow us to use a strategy of, very roughly speaking, matching 
asymptotic expansions and exact solutions.

In the remainder of this section we shall describe some basic properties
of the equation and its solutions. Section 2 discusses an illustrative
explicit example in which $f$ is piecewise linear, while Section 3 contains
the general results. { Apendices A and B cover
details of some basic asymptotics 
and spectral theoretical properties of \eqref{eq:1}.}

By an elementary Frobenius analysis (see Appendix 
\ref{section:ass}, for a rigorous 
summary of the main asymptotic results) the differential equation (\ref{eq:1}) is seen to
possess, for each $\lambda\in \Cc$, a unique (up to scalar multiples) solution 
in $L^2(-\pi,\pi)$, which we denote by $\phi(x,\lambda)$. { This} solution is continuous at $x=0$, where it is non-vanishing, and may therefore 
be normalized by the condition
\be \phi(0,\lambda) = 1. \label{eq:2} \ee
Any solution linearly independent of $\phi$ will blow up like $x^{-\pi/(2\von)}$ as
$x\rightarrow 0$.
Later we shall assume the more restrictive condition $0<\von<\pi/2$, which
is necessary to ensure that the differential equation has a unique solution
in a weighted $L^2$-space naturally associated with the problem.

Eqn. (\ref{eq:1}) also has singular points at $\pm \pi$; however all solutions
are square integrable at these points (in the unweighted space). In the weighted
space it will turn out that whether these solutions are square integrable or
not depends on how small $\von$ is.

By the {\em periodic eigenvalue problem associated with (\ref{eq:1})} we mean 
the problem of finding solutions which are square integrable and which satisfy
the periodicity condition $u(-\pi) = u(\pi)$. In view of the foregoing comments
about the solution $\phi$ it follows that $\mu$ is an eigenvalue if and only
if 
\be \phi(-\pi,\mu) = \phi(\pi,\mu). \label{eq:3} \ee
Now consider, for any $\lambda\in\Cc$, the function $\psi(x,\lambda) = \phi(-x,\lambda)$. 
It is easy to see that $\psi$ solves (\ref{eq:1}) with $\lambda$ replaced by 
$-\lambda$; also, $\psi$ is square integrable and $\psi(0,\lambda) = 1$, { hence}
$\psi(x,\lambda)=\phi(x,-\lambda)$. In other words,
\be \phi(-x,\lambda) = \phi(x,-\lambda). \label{eq:4} \ee
In a similar way, taking complex conjugates in (\ref{eq:1}) shows that
\be \overline{\phi(x,\lambda)} = \phi(x,-\overline{\lambda}) = \phi(-x,\overline{\lambda}). 
\label{eq:5} 
\ee
Eqn. (\ref{eq:5}) gives a reflection principle for $\phi$ across the imaginary
$\lambda$-axis, on which $\phi$ is real valued.

Using the eigenvalue condition (\ref{eq:3}) and the 
symmetries (\ref{eq:4}) and (\ref{eq:5}),
makes it clear that eigenvalues come in quadruples: $(\lambda,-\lambda,\overline{\lambda},
-\overline{\lambda})$. Also, the condition (\ref{eq:3}) for $\mu$ to be an eigenvalue
can be written in the equivalent form
\be \phi(\pi,\mu) = \phi(\pi,-\mu) = \overline{\phi(\pi,\overline{\mu})}. \label{eq:3p} \ee
We would like to show that all eigenvalues of the problem are real.

Finally in this introduction, we emphasize that all of our results can be placed in a 
proper operator theoretic framework: in particular, the eigenvalues to which we have
casually referred are indeed eigenvalues of a closed non-selfadjoint operator; see
Appendix B. Let $C^2_{\text{per}}(-\pi,\pi)$ be the space of periodic $C^2$ functions in $(-\pi,\pi)$. 
The differential expression $L$ with domain $C^2_{\text{per}}(-\pi,\pi)$
defines a closable linear operator acting on the space $L^2(-\pi,\pi)$,
despite the fact that $1/f(x)$ is not integrable at $x=0, \pm \pi$. 
In order to see this we will show that $L$ is a closed operator
in the maximal domain
\[
      \mathcal{D}_{\text m}:=\{u\in L^2(-\pi,\pi)\,:\, Lu\in L^2(-\pi,\pi)\}.
\]
Clearly $C^2_{\text{per}}(-\pi,\pi)\subset \mathcal{D}_{\text m}$. 
In Appendix~\ref{section:dom} we will show that
$\mathcal{D}_{\text m}\subset H^1(-\pi,\pi)$. 
This ensures that any eigenfunction of $\overline{(L,C^2_{\text{per}}(\pi,\pi))}$ 
is a solution of the periodic eigenvalue problem associated 
with (\ref{eq:1}), and vice versa. Note however that in the general case \eqref{eq:1} we do not claim to prove
  that the spectrum is purely discrete, though we believe this result
  is true -- see Corollary~\ref{lemma:dom1}.

\section{A piecewise linear example}\label{section:2}
As an example we start with the case where $f$ is given by
\be
 f(x) = \left\{ \begin{array}{ll} -2\pi^{-1}(x+\pi) & -\pi \leq x \leq -\pi/2, \\
                                   2\pi^{-1}x & -\pi/2 \leq x \leq \pi/2, \\
                                   2\pi^{-1}(\pi-x) & \pi/2 \leq x \leq \pi. 
                \end{array}\right. 
\label{eq:6} 
\ee
In this case the equation can be solved explicitly in terms of Bessel functions.
For $x\in [-\pi/2,\pi/2]$ we have
\be \phi(x,\lambda) = \Gamma(\nu+1)
(i\nu\lambda x)^{-\nu/2}J_{\nu}(2\sqrt{i\nu\lambda x}) =: \zeta_0(x,\lambda), \label{eq:7} \ee
where $\nu = \pi/(2\epsilon)$. For $x\in [\pi/2,\pi]$ we can write
\be \phi(x,\lambda) = A(\lambda)\zeta_1(x-\pi,\lambda) +  B(\lambda) \zeta_2(x-\pi,\lambda)
\label{eq:8}
\ee
where 
\be \zeta_1(z,\lambda) = (i\nu\lambda z)^{\nu/2}J_{\nu}(2\sqrt{-i\nu\lambda z}), \label{eq:8b}\ee
\be \zeta_2(z,\lambda) = (i\nu\lambda z)^{\nu/2}J_{-\nu}(2\sqrt{-i\nu\lambda z}), \label{eq:8c} \ee
and where $A$ and $B$ must be determined to ensure that $\phi$
is continuous and differentiable at $\pi/2$:
\be
\left(\begin{array}{cc} \zeta_1(-\pi/2,\lambda) & \zeta_2(-\pi/2,\lambda) \\
                        \zeta_1'(-\pi/2,\lambda) & \zeta_2'(-\pi/2,\lambda)
      \end{array}\right)\left(\begin{array}{c} A(\lambda) \\ B(\lambda)
                              \end{array}\right) 
 =  \left(\begin{array}{c} \zeta_0(\pi/2,\lambda) \\ \zeta_0'(\pi/2,\lambda)
                              \end{array}\right). 
\label{eq:match} 
\ee
In view of the fact that $\zeta_1(0,\lambda)=0$
the condition (\ref{eq:3p}) becomes, upon using (\ref{eq:8}), 
\[ B(\mu)\zeta_2(0,\mu) = B(-\mu)\zeta_2(0,-\mu). \]
The properties of the Bessel functions show that 
$\zeta_2(0,\lambda) = i^{\nu}/\Gamma(1-\nu)$ for
all $\lambda$, and so the eigenvalue condition becomes
\[ B(\mu) = B(-\mu). \]
The solution of (\ref{eq:match}) is given by
\be B(\lambda) = \frac{\det\left(\begin{array}{cc} \zeta_1(-\pi/2,\lambda) &\zeta_0(\pi/2,\lambda) \\
                                                   \zeta_1'(-\pi/2,\lambda) & \zeta_0'(\pi/2,\lambda)
      \end{array}\right) }{\det\left(\begin{array}{cc} \zeta_1(-\pi/2,\lambda) & \zeta_2(-\pi/2,\lambda) \\
                        \zeta_1'(-\pi/2,\lambda) & \zeta_2'(-\pi/2,\lambda)
      \end{array}\right) }. \label{eq:10} \ee

Further progress at this stage appears to rest on explicit calculation with Bessel functions. 
Firstly, it may be shown (either by reference to Watson's book on Bessel functions or
by an explicitly solving the first order ODE satisfied by the Wronskian) that
\be \det\left(\begin{array}{cc} \zeta_1(z,\lambda) & \zeta_2(z,\lambda) \\
                        \zeta_1'(z,\lambda) & \zeta_2'(z,\lambda)
      \end{array}\right) = \frac{\sin(\nu\pi)}{\pi}(i\nu\lambda)^\nu z^{\nu-1}. \label{eq:11}
\ee
Secondly, observe that if we define
\[ \chi(x,\lambda) = J_\nu(2\sqrt{i\nu\lambda x}), \]
then 
\[ \zeta_0(x,\lambda) = \Gamma(\nu+1)(i\nu\lambda x)^{-\nu/2} \chi(x,\lambda) = c\lambda^{-\nu/2}x^{-\nu/2}\chi(x,\lambda), \]
\[ \zeta_1(x,\lambda) = (i\nu\lambda x)^{\nu/2}\chi(-x,\lambda) = C\lambda^{\nu/2}x^{\nu/2}\chi(-x,\lambda). \]
An explicit calculation shows that
\[ \zeta_0'(\pi/2,\lambda) = c\lambda^{-\nu/2}(\pi/2)^{-\nu/2}(\chi'(\pi/2,\lambda)-\frac{\nu}{\pi}\chi(\pi/2,\lambda)), \]
\[ \zeta_1'(-\pi/2,\lambda) = C\lambda^{\nu/2}(-\pi/2)^{\nu/2}(-\chi'(\pi/2,\lambda)-\frac{\nu}{\pi}\chi(\pi/2,\lambda)), \]
and hence
\be \det\left(\begin{array}{cc} \zeta_1(-\pi/2,\lambda) &\zeta_0(\pi/2,\lambda) \\
                                                   \zeta_1'(-\pi/2,\lambda) & \zeta_0'(\pi/2,\lambda)
      \end{array}\right) = 2cC(-1)^{\nu/2}\chi(\pi/2,\lambda)\chi'(\pi/2,\lambda). \label{eq:12b} \ee
Since $\chi'(x,\lambda)=\sqrt{i\nu\lambda/x}J_\nu'(2\sqrt{i\nu\lambda x})$, it follows that
for some constant $\tilde{C}$ independent of $\lambda$,
\[ B(\lambda) =
\tilde{C}\lambda^{1/2-\nu}J_\nu'(\sqrt{2i\nu\lambda\pi})J_\nu(\sqrt{2i\nu\lambda\pi}). \]
Hence the eigenvalue condition becomes 
\be F(\lambda) = F(-\lambda), \;\;\;
 \text{where $F(\lambda)=\lambda^{1/2-\nu}J_\nu'(\sqrt{2i\nu\lambda\pi})J_\nu(\sqrt{2i\nu\lambda\pi})$}. \label{eq:12} \ee
Setting $z=\sqrt{2i\nu\lambda\pi}$, we observe that $\sqrt{-2i\nu\lambda\pi}=\pm i z$. 

Define
\be \rho(z) :=  \frac{J_\nu'(z)J_\nu(z)}{J_\nu'(iz)J_\nu(iz)} \label{eq:14}. \ee
Then a necessary condition for eqn. (\ref{eq:12}) to be satisfied becomes
\[
|\rho(z)| = 1.
\] 
\begin{theorem}
For any non-real $\lambda$ (corresponding to $\arg(z) \not\in \{ \pm \pi/4, \pm 3\pi/4\}$), 
the equation $|\rho(z)|=1$ cannot be satisfied. In fact, for 
$-\pi/4<\arg(\pm z)<\pi/4$, 
we have $|\rho(z)|<1$; for $\pi/4<\arg(\pm z)<3\pi/4$, we have $|\rho(z)|>1$.
Consequently all the eigenvalues of \eqref{eq:1}, with $f(x)$ given by \eqref{eq:6}, are real.
\end{theorem}
\begin{proof}
For real $\nu$, the zeros of the Bessel function $J_\nu$ and of its derivative
are all real. In the sector $-\pi/4 < \arg(z) < \pi/4$, if $z\neq 0$ then
$iz$ is not real and so { $J_\nu'(iz)J_\nu(iz)$} is nonzero. Consequently $\rho(z)$
is well defined and analytic. Standard asymptotics show that for large $|z|$ in 
this sector, $|\rho(z)|<1$. On the rays bounding the sector, $|\rho(z)|=1$. By 
the Phragmen--Lindel\"{o}f principle, since $\rho$ is a non-constant function which is 
analytic in a sector of angle $\pi/2$, and has a growth order strictly less than two (in fact, growth order 1), $\rho$ attains its maximum modulus strictly on the boundary rays. Consequently in the sector { $-\pi/4 < \arg(z) < \pi/4$}, one has
$|\rho(z)|<1$.

The results for the other sectors follow similarly; for instance, if
$-3\pi/4 < \arg(z) < -\pi/4$ then $z=-it$ with 
$-\pi/4 < \arg(t) < \pi/4$, and so $|\rho(z)| = 1/|\rho(t)|$. 
\end{proof}

\section{Results for the general case}
For the general case we return to our equation (\ref{eq:3p}). 
Throughout this section we assume $0<\von<\pi/2$.

We start with an auxiliary result concerning the location of the zeros
of $\phi(\pi,\lambda)$. 

\begin{theorem}\label{zeros}
All the zeros of $\phi(\pi,\lambda)$ lie on the negative imaginary axis.
\end{theorem}
\begin{proof}
Consider the differential equation
(\ref{eq:1}) satisfied by $\phi(x,\lambda)$, restricting attention to
the sub-interval $(0,\pi)$. Note that $f$ does not change sign in this 
sub-interval. Using standard integrating factor
techniques we can transform (\ref{eq:1}) { into} 
\be \frac{d}{dx}\left(p(x)\frac{du}{dx}\right) = -\frac{i\lambda}{\von}\frac{p}{f} u, 
\label{eq:1p}
\ee
where 
\be p'/p = f'/f + 1/(\von f). \label{eq:pdef} \ee

Simple asymptotic analysis shows that
\be p(x) \sim \left\{ \begin{array}{ll} x^{\pi/(2\von)}f(x) & \mbox{($x$ near zero);} \\
                                       (\pi-x)^{-\pi/(2\von)}f(x) & \mbox{($x$ near $\pi$).}
                     \end{array}\right.
\label{eq:pass}
\ee
Setting $w(x) = p(x)/f(x)$ and $\ell = i\lambda/\von$, we can write
(\ref{eq:1p}) in Sturm-Liouville form as
\be -(p(x)u')' = \ell w(x) u, \;\;\; x\in (0,\pi). \label{eq:1pp} \ee
Observe that the weight function $w$, which is positive, has very different behaviours
at different endpoints:
\[ w(x) \sim \left\{ \begin{array}{ll} x^{\pi/(2\von)} & \mbox{($x$ near zero);} \\
                                       (\pi-x)^{-\pi/(2\von)} & \mbox{($x$ near $\pi$).}
                     \end{array}\right.
\] 
By Frobenius analysis, there exist solutions of (\ref{eq:1}) with behaviours
$u(x)\sim x^{-\pi/(2\von)}$ and $u(x)\sim 1$ near $x=0$, and behaviours $u(x)\sim 1$ 
and $u(x)\sim (\pi-x)^{\pi/(2\von)}$ near $x=\pi$. It is therefore easy to check that 
for $0<\von<\pi/2$ there is always precisely one solution in $L^2_w$ near the origin, 
one solution in $L^2_w$ near $x=\pi$. The problem (\ref{eq:1pp}) is therefore of limit-point type { at the two ends of the interval $(0,\pi)$}.

Since { $p>0$ and $w>0$} on $(0,\pi)$, (\ref{eq:1pp}) is automatically 
selfadjoint in $L^2_w$ without the need for boundary conditions.
The eigenvalues $\ell$, which must be real, will be precisely those values of 
$\ell$ for which $\phi(x,\lambda)$, which is always in $L^2_w$ near the origin,
is in $L^2_w$ near $x=\pi$.  This only happens when $\phi(x,\lambda)\sim (\pi-x)^{\pi/(2\von)}$, that is, when $\phi(\pi,\lambda)=0$. 
Since $\ell = i\lambda/\von$, this establishes that
the zeros of $\phi(\pi,\lambda)$ lie on the imaginary axis in the $\lambda$-plane.
Since the eigenvalues of (\ref{eq:1pp}) are all strictly
positive, we can further say that the zeros of $\phi(\pi,\lambda)$ all lie on the
negative imaginary $\lambda$-axis.
\end{proof}

\begin{remark}
This result, which we have proved for $0<\von<\pi/2$, is also true for
$0<\von<\pi$. In the case $\pi/2\leq \von < \pi$ the problem is limit circle
in $L^2_w$ at $x=\pi$ but the condition $\phi(\pi,\lambda)=0$ arises naturally
by imposing $u(\pi)=0$ in (\ref{eq:1pp}), which is the Friedrichs boundary
condition for this case.
\end{remark}

We now make the transformation $\lambda = iz^2$ and define
\be g(z) = \phi(\pi,-iz^2). \label{eq:gdef} \ee
The eigenvalue condition $\phi(\pi,\lambda)=\phi(\pi,-\lambda)$ thus becomes
\be g(z) = g(iz). \label{eq:evc} 
\ee
In the $\lambda$-plane, by Theorem \ref{zeros}, all the zeros of 
$\phi(\pi,\lambda)$ are of the form $\lambda = -ir$, $r>0$; hence in the $z$-plane
all the zeros of $g$ lie on the imaginary $z$-axis and occur in pairs
$\pm i\alpha_n$, $n\in \Nn$, $0<\alpha_1\le\alpha_2\le\dots$ (observe that $g(0)\neq 0$). When $z$ is not one of these
zeros the function
\be \rho(z) = \frac{g(iz)}{g(z)} \label{eq:rhodef} \ee
is well defined and analytic. The eigenvalue condition (\ref{eq:evc}) becomes
\be \rho(z) = 1. \label{eq:evcp} \ee
\begin{theorem}
\label{maxp}
Suppose that the solution $\phi(x,\lambda)$ is, for each $x\in [0,\pi]$, an
analytic function of $\lambda$ with growth order at most 1/2. Then
for $-\pi/4 < \arg(\pm z) < \pi/4$, $|\rho(z)|<1$; for $\pi/4 < \arg(\pm z)
 < 3\pi/4$,  { $|\rho(z)|>1$}. Consequently, eigenvalues of \eqref{eq:1} can only lie on the lines
$\arg(z) \equiv \pi/4 \; ({\rm mod} \; \pi/2)$, which, in terms of $\lambda = iz^2$,
means that all eigenvalues are real.
\end{theorem}

\begin{proof}
We start by outlining the idea of the proof; after the outline, we fill
in a technical detail.

Firstly we know that $g(0)\neq 0$, that $g$ is an analytic function of $z^2$ 
and that the zeros of $g$ occur in pairs $\pm i\alpha_n$. If these zeros have 
the appropriate asymptotic behaviour then we can write
\be g(z) = \exp(h(z^2))\prod_{n=1}^\infty\left(1+\frac{z^2}{\alpha_n^2}\right) 
\label{eq:gprod}
\ee
where $h$ is an analytic function and the infinite product is convergent.
Since we also know that $g$ has exponential growth order at most 1, it
follows that $h$ must be constant, and hence
\[ \rho(z) = \frac{g(iz)}{g(z)}
 = \prod_{n=1}^\infty\left(\frac{1-\frac{z^2}{\alpha_n^2}}{1+\frac{z^2}{\alpha_n^2}}\right).
\]
It is then a simple matter to check that for each $n\in \Nn$,
\[ \left|\frac{1-\frac{z^2}{\alpha_n^2}}{1+\frac{z^2}{\alpha_n^2}}\right|
 < 1 \qquad \forall z \text{ such that } -\pi/4 < \arg(z) < \pi/4, \]
and the result is immediate.

It only remains to check that the zeros of $g$ grow sufficiently rapidly to ensure
convergence of the infinite product in (\ref{eq:gprod}). To this end we
recall from the proof of Theorem \ref{zeros} that the zeros of $\phi(\pi,\lambda)$
are given by the eigenvalues of the ODE (\ref{eq:1p}) on $(0,\pi)$ with Friedrichs 
boundary conditions. We therefore wish to examine the large eigenvalues of this problem. 
To this end we make the transformation
\be \phi(x,\lambda) = \frac{1}{\sqrt{p(x)}}\psi(x,\lambda). \label{eq:psidef} \ee
where $p$ is given by (\ref{eq:pdef}). This reduces the differential equation
to
\be -\psi'' + Q(x)\psi = \mu\frac{1}{f}\psi, \label{eq:psieq} \ee
where $Q = p^{-1/2}(p^{-1/2}p')'/2$ and where $\mu = i\von^{-1}\lambda$ is already
known to be real and positive. A simple calculation based on the known
asymptotics of $p$ near $x=0$ and $x=\pi$ shows that
\[ Q(x) \sim \left\{
 \begin{array}{ll}
  \frac{1}{4}\left(\frac{\pi}{2\von}+1\right)\left(\frac{\pi}{2\von}-1\right)x^{-2}, & 
 x \searrow 0, \\
\frac{1}{4}\left(\frac{\pi}{2\von}+1\right)\left(\frac{\pi}{2\von}-1\right)(\pi-x)^{-2}, & 
 x \nearrow \pi. \end{array} \right. 
  \]
Since $0<\von<\pi/2$, we have $Q(x)\rightarrow +\infty$ for $x\searrow 0$ and
$x\nearrow \pi$. By Sturm Comparison we can therefore obtain lower bounds on 
the eigenvalues $\mu$ by discarding $Q$ and considering the eigenvalues $\mu$ { 
of the problem
\[ -\psi''-\tilde{Q}\psi = \mu\frac{1}{f(x)}\psi, \;\;\; x\in (0,\pi), \]
where $\tilde{Q}>0$ is a sufficiently large constant.}
Standard WKB estimates (see, e.g., Bender and Orszag \cite{Bender})
give for the $n$th eigenvalue the formula
\[ \int_0^\pi \sqrt{\frac{\mu_n}{f(x)}}dx \sim n, \]
where we observe that the integral converges because $f$ has simple zeros at the endpoints
of the interval. Hence
\[ \mu_n = O(n^2). \]
Recalling the transformations $\lambda = iz^2$ and $\lambda = i\von^{-1}\mu$ and
the definition of the zeros $\alpha_n$ of $g$ and their correspondence with the
eigenvalues of this problem, we must have
\[ \lambda_n = i\alpha_n^2 = i\von^{-1}\mu_n, \]
whence
\[ \alpha_n = O(n). \]
This is sufficient to ensure convergence of the product in (\ref{eq:gprod}).
\end{proof}

\begin{remark} \label{remark:growthorder1}
In the case of the piecewise linear coefficient example of eqn. (\ref{eq:6}),
$\phi(x,\lambda)$ has growth order $1/2$ in $\lambda$ and hence
$\phi(x,-iz^2)$ has growth order 1 in $z$, for any $x \in [0,\pi]$.
This follows from the explicit expressions for the solutions given in
terms of Bessel functions in Section \ref{section:2}.

The same conclusion also holds for the case $f(x) = (2/\pi)\sin(x)$ considered
by Davies \cite{Davies}: the equation admits a change of variables $t=\tan(x/2)$, and the solutions are then given in this case in terms
of Heun G-functions on the interval $[0, +\infty)$. For relevant asymptotic results, see e.g. 
\cite{Fedoryuk}.

In both of these cases, therefore, the eigenvalues of the corresponding
PT-symmetric eigenvalue problem are all real.
\end{remark}

The following lemma widens the class of functions $f$ for which the 
function $\phi(x,\lambda)$ has growth order $1/2$ in $\lambda$ for
all $x\in [0,\pi]$.

\begin{lemma}\label{lemma:growthorder2}
Suppose that $f$ falls within the class introduced in Section
\ref{section:1} and that $f$ is linear in arbitrarily small
neighbourhoods of $x=0$, $x=\pi$. Then $\phi(x,\lambda)$ has
growth order $1/2$ in $\lambda$ for all $x\in [0,\pi]$.
\end{lemma}
\begin{proof}
Suppose that $f(x) = 2x/\pi$ for $x\in [0,\delta]$ and that
$f(x) = 2(\pi-x)/\pi$ for $x\in [\pi-\delta,\pi]$ for some
$\delta>0$. Then the explicit expression for $\phi$
given in (\ref{eq:7}) is valid for all $x\in [0,\delta]$ 
and, with standard asymptotics of Bessel functions, establishes
that $\phi(x,\lambda)$ and $\phi'(x,\lambda)$ have exponential
growth order $1/2$ in $\lambda$ for all such $x$. Writing
the differential equation for $\phi$ as
\be \frac{d}{dx}\left(\hspace{-1mm}\begin{array}{c}\phi(x,\lambda) \\
 \lambda^{-1/2}f(x)\phi'(x,\lambda)\end{array}\hspace{-1mm}\right)
 = \left(\hspace{-1mm}\begin{array}{cc} 0 & \lambda^{1/2}/f(x) \\
  -i\lambda^{1/2}/\von & -1/(\von f) \end{array}
 \hspace{-1mm}\right) \left(\hspace{-1mm}\begin{array}{c}\phi(x,\lambda) \\
 \lambda^{-1/2}f(x)\phi'(x,\lambda)\end{array}\hspace{-1mm}\right)
\label{eq:system} \ee
we observe that this is a system of the form $Y' = A(x,\lambda)Y$
in which $\| A(x,\lambda) \| \leq C_{\delta}|\lambda|^{1/2}$
for all $x\in [\delta,\pi-\delta]$. By standard Picard
estimates 
\[ \left\| \left(\begin{array}{c}\phi(x,\lambda) \\
 \lambda^{-1/2}f(x)\phi'(x,\lambda)\end{array}\right)\right\|
 \leq
 \left \| \left(\begin{array}{c}\phi(\delta,\lambda) \\
 \lambda^{-1/2}f(\delta)\phi'(\delta,\lambda)\end{array}\right)\right \|
 \exp(C_\delta(x-\delta)|\lambda|^{1/2}) \]
for all $x\in [\delta,\pi-\delta]$. This establishes that 
the result of our lemma holds for all $x\in [0,\pi-\delta]$.
Finally we must extend the result over $[\pi-\delta,\pi]$. To 
this end we write
\be \phi(x,\lambda) = A(\lambda)\zeta_1(x-\pi,\lambda) 
 + B(\lambda)\zeta_2(x-\pi,\lambda), \;\;\; x\in [\pi-\delta,\pi] 
\label{eq:phiend} \ee
where $\zeta_1$ and $\zeta_2$ are the functions introduced
in eqns. \eqref{eq:8b}, \eqref{eq:8c}, both of which have growth
order $1/2$ in $\lambda$ for all $x$. The coefficients may be
found by solving the system
\be
\left(\begin{array}{cc} \zeta_1(-\delta,\lambda) & \zeta_2(-\delta,\lambda) \\
                        p(\delta)\zeta_1'(-\delta,\lambda) & p(\delta)\zeta_2'(-\delta,\lambda)
      \end{array}\right)\left(\begin{array}{c} A(\lambda) \\ B(\lambda)
                              \end{array}\right) 
 =  \left(\begin{array}{c} \phi(\delta,\lambda) \\ p(\delta)\phi'(\delta,\lambda)
                              \end{array}\right) 
\label{eq:match2} 
\ee
where  $p$ may be any function which is non-zero at
$\delta$ but is most conveniently chosen according to (\ref{eq:pdef}).
{ This ensures that the determinant of the matrix on the left
hand side of (\ref{eq:match2}) is independent of $\delta$ and may be
calculated from its value as $\delta\rightarrow 0$. From the asymptotics
of $\zeta_1$ and $\zeta_2$ it turns out that this value is a non-zero
constant, and so}
\[ \left(\begin{array}{c} A(\lambda) \\ B(\lambda)
                              \end{array}\right) 
 = { c(\lambda)}\left(\begin{array}{cc} p(\delta)\zeta_2'(-\delta,\lambda) & -\zeta_2(-\delta,\lambda) \\
                        -p(\delta)\zeta_1'(-\delta,\lambda) & \zeta_1(-\delta,\lambda)
      \end{array}\right)\left(\begin{array}{c} \phi(\delta,\lambda) \\ p(\delta)\phi'(\delta,\lambda)
                              \end{array}\right) 
\]
Every quantity appearing on the
right hand side has growth order {  at most} $1/2$ in $\lambda$ and so $A(\lambda)$ and 
$B(\lambda)$ have growth order $1/2$ in $\lambda$. Thus from (\ref{eq:phiend})
it follows that $\phi(x,\lambda)$ has growth order $1/2$ for all $x\in [\pi-\delta,
\pi]$.
\end{proof}
\begin{remark}\label{remark:growthorder3}
The ideal way to obtain our final result now would be to show that 
$\phi(x,\lambda)$ always has growth order $1/2$ for the whole class of
coefficients $f$ introduced in Section \ref{section:1}. Unfortunately
this does not appear to be easy, since growth orders are not necessarily
preserved in approximation limits. For instance, $\exp(\lambda^2)$ has
growth order $2$ but can be approximated locally uniformly by Taylor
polynomials, which all have growth order $0$. Nevertheless the following
result is true.
\end{remark}

\begin{theorem}\label{theorem:final}
Let $f$ be a function of the class introduced in Section \ref{section:1}
which additionally has the property that $f''(0)$ and $f''(\pi)$
exist. Then all the eigenvalues of eigenvalue 
problem \eqref{eq:1} are real.
\end{theorem}
\begin{proof} By Theorem \ref{maxp} and Lemma \ref{lemma:growthorder2}
the result holds for functions $f$ which are linear in arbitrarily
small neighbourhoods of the endpoints.

Now suppose that we have a problem for which $f$ is not linear near
the endpoints. Nevertheless we may approximate $f$ by a function which
is linear on $[0,\delta]$ and $[\pi-\delta,\pi]$ for small $\delta$.
From \eqref{eq:3} and \eqref{eq:4}, we know that the condition for
$\mu$ to be an eigenvalue is that $\mu$ be a zero of the function
\[ d(\lambda) = \phi(\pi,\lambda) - \phi(\pi,-\lambda). \]
We shall attach subscripts $\delta$ to the quantities where $f$ is
replaced by a linear function on $[0,\delta]\cup [\pi-\delta,\pi]$:
thus we shall approximate $\phi(x,\lambda)$ by $\phi_{\delta}(x,\lambda)$
and $d(\lambda)$ by $d_\delta(\lambda)$. If we can show that
\be \lim_{\delta\searrow 0}(d_\delta(\lambda)-d(\lambda))=0 \label{eq:dcon}\ee
locally uniformly in $\lambda$, then the zeros of $d(\lambda)$
will be precisely the limits of the zeros of $d_\delta(\lambda)$.
{ This is achieved, for instance,} by using the argument principle to count zeros inside any
contour, exploiting the fact that for analytic functions the
locally uniform convergence of the functions implies locally
uniform convergence of their derivatives. Since the zeros of
$d_\delta$ all lie on the real axis, we shall have proved
our result.

Choose $f_\delta$ as follows. On $[0,\delta]\cup [\pi-\delta,\pi]$
$f_\delta$ should be linear and it should also match the values
of $f$ at $0$, $\delta$, $\pi-\delta$ and $\pi$: consequently
we shall have
\[ f_\delta(0)=f(0) = 0; \;\; f_\delta(\delta) = f(\delta); \;\;
 f_\delta(\pi-\delta)=f(\pi-\delta); \;\; f_\delta(\pi) = f(\pi) = 0; \]
also, as $\delta\searrow 0$, 
\[ f_\delta'(0) \rightarrow f'(0); \;\;\;
 f_\delta'(\pi) \rightarrow f'(\pi). \]
Frobenius analysis (more precisely, the uniform 
asymptotics in Lemma \ref{lemma:ass4}) shows that 
\be \phi(x,\lambda) = 1-\frac{i\lambda}{1+\von f'(0)} x + o(x); \;\;\;
   \phi'(x,\lambda) = -\frac{i\lambda}{1+\von f'(0)} + o(1) \label{eq:asscompare1} \ee
for small $x$. It is convenient also to have the
second solution $\Phi(x,\lambda)$ which satisfies
\be \Phi(x,\lambda) = x^{-1/(\von f'(0))}(1+O(x)), \;\;\;
   \Phi'(x,\lambda) = \frac{-1}{\von f'(0)}x^{-1-1/(\von f'(0))}(1+O(x)) \label{eq:asscompare2} \ee
for small $x$. 

Note that we also have similar asymptotics for $\phi_\delta$ and $\Phi_\delta$:
\[ \phi_\delta(x,\lambda) = 1-\frac{i\lambda}{1+\von f_\delta'(0)} x + o(x); \;\;\;
   \phi_\delta'(x,\lambda) = -\frac{i\lambda}{1+\von f_\delta'(0)} + o(1); \]
\[ \Phi(x,\lambda) = x^{-1/(\von f_\delta'(0))}(1+O(x)), \;\;\;
   \Phi'(x,\lambda) = -\frac{1}{\von f_\delta'(0)}x^{-1-1/(\von f_\delta'(0))}(1+O(x)). \]
Moreover the { correction terms can be bounded locally
uniformly in $x\in [0,\delta]$} by the remarks following
Lemma \ref{lemma:ass4}. Consequently, since $f_\delta'(0)-f'(0) = o(1)$,
\[ \phi(\delta,\lambda)-\phi_\delta(\delta,\lambda)
 = o(\delta); \;\;\;
\phi'(\delta,\lambda)-\phi_\delta'(\delta,\lambda)
 = o(1). \]
All of these bounds are locally uniform in $\lambda$.

Now on $[\delta,\pi-\delta]$ we have
\be \phi_\delta(x,\lambda) = c_1 \phi(x,\lambda) + c_2\Phi(x,\lambda) \label{eq:phidel} \ee
where the coefficients $c_1$ and $c_2$ are chosen to ensure continuity
and differentiability at $x=\delta$. In terms of the coefficient $p$
introduced in (\ref{eq:pdef}), and suppressing the $\lambda$-dependence
for simpler notation, we have
\[ \left(\begin{array}{cc} \phi(\delta) & \Phi(\delta)
 \\ p\phi'(\delta) & p\Phi'(\delta) \end{array}\right)
 \left(\begin{array}{c} c_1 \\ c_2 \end{array}\right)
 =\left(\begin{array}{c} \phi_\delta(\delta) \\ p\phi_\delta'(\delta) \end{array}\right), \]
whence, in terms of the determinant $\Delta$ of the matrix on the left
hand side of this system,
\[ \left(\begin{array}{c} c_1 \\ c_2 \end{array}\right)
 = \left(\begin{array}{c} 1 \\ 0 \end{array}\right) 
 + \frac{p(\delta)}{\Delta}\left(\begin{array}{c}
 \Phi'(\delta)(\phi_\delta(\delta)-\phi(\delta))
 - (\phi_\delta'(\delta)-\phi'(\delta))\Phi(\delta) \\
 \phi(\delta)(\phi_\delta'(\delta)-\phi'(\delta))
 - \phi'(\delta)(\phi_\delta(\delta)-\phi(\delta)) \end{array}\right). \]
Bearing in mind that from (\ref{eq:pass}), 
\[ p(\delta) = O(\delta^{1+1/(\von f'(0))}) \]
and that $\Delta$ is a $\delta$-independent constant, we obtain by
elementary estimates that
\[ c_1 = 1 + o(\delta); \;\;\; c_2 = o(\delta^{1+1/(\von f'(0))}). \]
From (\ref{eq:phidel}) it follows that for each fixed $x\in [0,\pi-\delta]$
we have
\[ \lim_{\delta\searrow 0}\phi_\delta(x,\lambda) = \phi(x,\lambda) \]
locally uniformly in $\lambda$. We need to extend this convergence
to $x=\pi$.

To this end we introduce the solutions $\psi(x)$ and $\Psi(x)$ for
the unperturbed system and $\psi_\delta(x)$ and $\Psi_\delta(x)$ for
the perturbed system, determined by the asymptotic behaviours
(see Lemma \ref{lemma:ass4} again)
\be \psi(x) = 1+\frac{i\lambda}{1-\von f'(\pi)}(\pi-x)+o(\pi-x), \;\;\;
   \Psi(x) = (\pi-x)^{1/(\von f'(\pi))}(1+O(\pi-x)),
 \label{eq:asscompare3} \ee
\be \psi_\delta(x) = 1+\frac{i\lambda}{1-\von f_\delta'(0)}(\pi-x)+o(\pi-x), \;\;\;
   \Psi_\delta(x) = (\pi-x)^{1/(\von f_\delta'(\pi))}(1+O(\pi-x)).
 \label{eq:asscompare4} \ee
We write
\[ \phi_\delta(x) = c_{1,\delta}\psi_\delta(x) +
  c_{2,\delta}\Psi_\delta(x) \]
and 
\[ \phi(x) = c_{1}\psi(x) +
  c_{2}\Psi(x). \]
In view of the asymptotics it follows that
\[ \phi(\pi) = c_1; \;\;\; \phi_{\delta}(\pi) = c_{1,\delta}. \]
Therefore in order to establish that $\lim_{\delta\searrow 0}\phi_{\delta}(\pi,\lambda) 
 = \phi(\pi,\lambda)$
locally uniformly in $\lambda$ it is sufficient for us to prove 
that $c_{1,\delta}$ converges locally uniformly to $c_1$. 
Now
\[ \left(\begin{array}{cc} \psi_\delta(\pi-\delta) &
 \Psi_\delta(\pi-\delta) \\
 p\psi_\delta'(\pi-\delta) & p\Psi_\delta'(\pi-\delta)\end{array}\right)
 \left(\begin{array}{c} c_{1,\delta} \\ c_{2,\delta}\end{array}\right)
 = 
 \left(\begin{array}{c} \phi_\delta(\pi-\delta) \\ 
 p\phi_\delta'(\pi-\delta) \end{array}\right), \]
and a similar equation holds with $\delta=0$. 
A calculation similar to the ones performed before shows that
\[ { \begin{array}{lcl} c_{1,\delta} & = & c_1 + {\displaystyle \frac{p}{\Delta}
 \left\{  (\Psi_\delta'-\Psi')\phi_\delta
 + \Psi'(\phi_\delta-\phi) \right. } \\
 & & \\
 & & {\displaystyle \left. \left. - (\Psi_\delta-\Psi)\phi_\delta'
 - \Psi(\phi_\delta'-\phi') \right\}\right|_{x=\pi-\delta} . } \end{array}}\]
We now estimate the various terms, bearing in mind that from
(\ref{eq:pass}) we have 
\[ p(\pi-\delta) = O(\delta^{1-1/(\von f'(\pi))}) \]
which may be large. Consider, for example, the term
\[ p(\pi-\delta)(\Psi_\delta'(\pi-\delta)-\Psi'(\pi-\delta))\phi_\delta(\pi-\delta). \]

In this term we have
\[ \Psi_\delta'(\pi-\delta)-\Psi'(\pi-\delta) = O(\delta^{1/(\von f'(\pi))-1}
 (\delta^{1/(\von f_\delta'(\pi))-1/(\von f'(\pi))}-1)). \]
The multiplicative factor of $\delta^{1/(\von f'(\pi))-1}$
cancels with the $p(\pi-\delta)$ term, while the $\phi_\delta(\pi-\delta)$
term is close to $\phi(\pi-\delta)$ which remains bounded as
$\delta\searrow 0$, since all solutions of the unperturbed equation
are bounded in a neighbourhood of $\pi$. This leaves a term
\[ O(\delta^{\nu} - 1) \]
in which
\[ \nu = \von^{-1}\frac{f'(\pi)-f_\delta'(\pi)}{f_\delta'(\pi)
 f'(\pi)}. \]
However the fact that $f$ is twice differentiable at $\pi$ ensures from
the construction of $f_\delta$ that
\[ f_\delta'(\pi) - f'(\pi) = \frac{\delta}{2}f''(\pi) + o(1) \]
and so the term which is $ O(\delta^{\nu} - 1) $ is actually
\[ O(\delta^{\delta f''(\pi)/2}-1). \]
This tends to zero as $\delta$ tends to zero, regardless of whether
$f''(\pi)$ be positive or negative. 

The other terms may be dealt with in a similar way. Since all the
asymptotics are locally uniform in $\lambda$, we get
\[ \lim_{\delta\searrow 0}\phi_\delta(\pi,\lambda) =  \phi(\pi,\lambda) \]
locally uniformly with respect to $\lambda$. The result follows.
\end{proof}

\appendix
\section{Appendix: Asymptotics}\label{section:ass}
In this Appendix we examine the asymptotic behaviour of solutions $u$ 
of (\ref{eq:1}) and their derivatives, separately near $x=0$ and $x=\pi$. 
In order to have a completely rigorous treatment in each case we shall 
transform the problem to an infinite interval and invoke the
Levinson Theorem \cite[Theorem 1.3.1]{kn:Eastham}.

Consider first the behaviour in a neighbourhood of $0$. Make the
change of variable $x=\exp(-t)$ so that $0$ maps to $\infty$. Eqn.
(\ref{eq:1}) now becomes
\be i\von \frac{d}{dt}\left(\mbox{e}^tf(\mbox{e}^{-t})\frac{du}{dt}\right)
 - i\frac{du}{dt} = \lambda \mbox{e}^{-t}u. \label{eq:ax1} \ee
Under our hypotheses on $f$ we know that for all $x$ { in a neighbourhood of 0}, 
\[ f(x) = x f'(0) + r(x), \]
where $|r(x)|\leq Cx^2$ for some constant $C>0$. As a consequence
\[ \mbox{e}^tf(\mbox{e}^{-t}) = f'(0) + \rho(t), \]
where
\be |\rho(t)| \leq C\exp(-t). \label{eq:ax2} \ee
Consequently we can write (\ref{eq:ax1}) as a first order system
\be \frac{d}{dt}\left(\begin{array}{c} u \\
 (f'(0) + \rho(t))u'\end{array} \right) 
 = \left(\begin{array}{cc} 0 & \frac{1}{f'(0)+\rho(t)} \\
 \frac{-i\lambda}{ \von}\mbox{e}^{-t} & \frac{\von^{-1}}{f'(0)+\rho(t)} \end{array}\right)
 \left(\begin{array}{c} u \\
 (f'(0) + \rho(t))u'\end{array} \right) \label{eq:ax3} \ee
or as
\be \frac{d{\bf u}}{dt} = (A_\infty + R(t)){\bf u}, \label{eq:ax4} \ee
where 
\[ {\bf u} = \left(\begin{array}{c} u \\
 (f'(0) + \rho(t))u'\end{array} \right), \]
\[ A_\infty = \left(\begin{array}{cc} 0 & 1/f'(0) \\
 0 & \von^{-1}/f'(0) \end{array}\right), \]
and 
\be \| R(t) \| \leq C \mbox{e}^{-t}. \label{eq:ax3.5} \ee
The matrix $A_\infty$ has eigenvalues $0$ and $(\von f'(0))^{-1}$ and is
diagonalizable:
\[  A_\infty = V \Lambda V^{-1}, \]
where $\Lambda = \mbox{diag}(0,(\von f'(0))^{-1})$, 
\[ V = \left(\begin{array}{cc} 1 & \von \\
 0 & 1 \end{array}\right). \] 
Consequently the system (\ref{eq:ax4}) can be diagonalized by the transformation
$ {\bf v} = V^{-1}{\bf u}$, and
\be \frac{d{\bf v}}{dt} = (\Lambda + V^{-1}R(t)V){\bf v}. \label{eq:ax5}\ee

\begin{lemma}\label{lemma:ass0}
Suppose that $T$ is chosen sufficiently large to ensure that 
\[ \int_T^\infty \| V^{-1} R(t) V \| dt < 1/2. \]
Then the system (\ref{eq:ax5}) has solutions ${\bf v}_1$ and
${\bf v}_2$ such that
\be {\bf v}_1(t) = { \begin{pmatrix} 1\\0 \end{pmatrix}} + {\bf r}_1(t), \label{eq:ax6} \ee
\be {\bf v}_2(t) = \exp(t/(\von f'(0)))\left\{{ \begin{pmatrix} 0\\1 \end{pmatrix}} + {\bf r}_2(t)\right\}, \label{eq:ax7} \ee
where, for all $t \geq T$, 
\be \max(\| {\bf r}_1(t) \|, \| {\bf r}_2(t) \|)
 \leq 2\int_t^\infty \| V R(s) V^{-1} \| ds . \label{eq:ax8} \ee
Consequently, in view of (\ref{eq:ax3.5}) and the transformation
$ {\bf v} = V^{-1}{\bf u}$, the linear system (\ref{eq:ax4}) possesses 
solutions ${\bf u}_1$ and ${\bf u}_2$ such that
\be {\bf u}_1(t) = V{ \begin{pmatrix} 1\\0 \end{pmatrix}} + {\bf \eta}_1(t), \label{eq:ax9} \ee
\be {\bf u}_2(t) = \exp(t/(\von f'(0)))\left\{V{ \begin{pmatrix} 0\\1 \end{pmatrix}} + \eta_2(t)\right\}, \label{eq:ax10} \ee
where, for some constant $C>0$ and for all $t \geq T$, 
\be \max(\| \eta_1(t) \|, \| \eta_2(t) \|)
 \leq { C \exp(-t).} \label{eq:ax11} \ee
\end{lemma}
\begin{proof} 
This is just the statement of Levinson's Theorem for our problem.
\end{proof}
It is now a simple matter to translate this lemma into information about
$u(x)$, $u'(x)$ and $f(x)u'(x)$ for $x$ in a neighbourhood of $0$:
\begin{lemma}\label{lemma:ass1}
For the equation (\ref{eq:1}), there exist solutions $u_1$ and $u_2$
such that
\be 
 \left(\begin{array}{c} 
 u_1(x) \\ -f(x)u_1'(x) \end{array}\right) 
 = \left(\begin{array}{c} 1 \\ 0 \end{array}\right) + \eta_1(x), \label{eq:ax12} \ee
\be 
 \left(\begin{array}{c} 
 u_2(x) \\ -f(x)u_2'(x) \end{array}\right) 
 = x^{-1/(\von f'(0))}\left\{ \left(\begin{array}{c} \von  \\ 1 \end{array}\right) 
 + \eta_2(x)\right\}, \label{eq:ax13} \ee
where $\| \eta_1(x) \| \leq Cx$, $\| \eta_2(x) \| \leq Cx$ for all $0\leq x \leq x_0$,
where $x_0$ is some positive real number. Moreover the error bounds are locally
uniform in $\lambda$ and depend only on $\lambda$, $\varepsilon$ and { $\sup_{0<x<x_0}|r(x)|$}.
\end{lemma}
Observe that the solution $u_1$ is the solution $\phi$ which is regular near zero,
used throughout our article. 

A similar lemma can be proved concerning behaviour near $x=\pi$:
\begin{lemma}\label{lemma:ass2}
For the equation (\ref{eq:1}), there exist solutions $u_3$ and $u_4$
such that
\be 
 \left(\begin{array}{c} 
 u_3(x) \\ -f(x)u_3'(x) \end{array}\right) 
 = \left(\begin{array}{c} 1 \\ 0 \end{array}\right) + \eta_3(x), \label{eq:ax14} \ee
\be 
 \left(\begin{array}{c} 
 u_4(x) \\ -f(x)u_4'(x) \end{array}\right) 
 = (\pi-x)^{1/(\von f'(0))}\left\{ \left(\begin{array}{c} -\von  \\ 1 \end{array}\right) 
 + \eta_4(x)\right\}, \label{eq:ax15} \ee
where $\| \eta_1(x) \| \leq C(\pi-x)$, $\| \eta_2(x) \| \leq C(\pi-x)$ for all $\pi-x_0 \leq x 
\leq \pi$, where $x_0$ is some positive real number. Moreover the error bounds are 
locally uniform in $\lambda$ and depend only on $\lambda$, $\varepsilon$ and { $\sup_{-x_0<x<0}|r(x)|$}.
\end{lemma}

These results are sufficient for most of our article, but not quite for the estimates
in the proof of Theorem \ref{theorem:final}. In particular, although Lemma \ref{lemma:ass1}
gives the results (\ref{eq:asscompare2}) and (\ref{eq:asscompare4}), it does not quite
give the results (\ref{eq:asscompare1}) and (\ref{eq:asscompare3}), since the terms
$\frac{i\lambda}{1+\von f'(0)}x$ are absorbed into the error terms. 
The solution to this
problem relies on using an integral equation, eqn. (1.4.13) in \cite[p. 12]{kn:Eastham},
according to which the solution ${\bf v}_1$ of Lemma \ref{lemma:ass0} will satisfy
\be {\bf v}_1(t) = \left(\begin{array}{c} 1 \\ 0 \end{array}\right)
 - \int_t^\infty { F(t)F^{-1}(s)}V^{-1}R(s)V {\bf v}_1(s)ds, 
 \label{eq:vbetter} \ee
in which, for our system,
\be { F(t)} = \left(\begin{array}{cc} 1 & 0 \\ 0 & \exp(t/(\von f'(0))) \end{array}\right). 
\label{eq:ass17} \ee
The higher quality approximation is achieved by replacing ${\bf v}_1$ in the integral
on the right hand side by its approximation from Lemma \ref{lemma:ass0}, so that 
\be V^{-1}R(s)V {\bf v}_1(s) = V^{-1}R(s)V \left(\begin{array}{c} 1 \\ 0 \end{array}\right) + O(\exp(-2s)). 
 \label{eq:ass18} \ee
 
Suppose that $f''(0)$ and $f''(\pi)$ exist. Some explicit calculations show that
\[ V^{-1}R(s) V = 
 \mbox{e}^{-s}\left(\begin{array}{cc} \lambda i & \lambda i \von \\
 -\lambda i \von^{-1} & -\lambda i - f''(0)/{ (2\von f'(0)^2)} \end{array}\right) +   \mbox{e}^{-s} { \begin{pmatrix} 0 & 0 \\ 0 & g(s)  \end{pmatrix}}\]
where 
\be { g(t) = O(\mbox{e}^t(\mbox{e}^tf(\mbox{e}^{-t})-f'(0))-f''(0)/2) \label{eq:gdef2}} \ee
tends to zero as $t\rightarrow\infty$. Hence
\be  V^{-1}R(s)V {\bf v}_1(s) = \mbox{e}^{-s}\lambda i \left(\begin{array}{c}
 1 \\ -\von^{-1} \end{array}\right) +O(\mbox{e}^{-2s}). \label{eq:ass19}\ee
Substituting (\ref{eq:ass17}) and (\ref{eq:ass19}) back into (\ref{eq:vbetter}) and transforming
back to the original variables yields the following result:
\begin{lemma}\label{lemma:ass4} The solution $u_1$ of (\ref{eq:1}) mentioned in Lemma
\ref{lemma:ass1} has an asymptotic form
\be \left(\begin{array}{c} u_1(x) \\ -f(x)u_1'(x) \end{array}\right)
 = \left(\begin{array}{c} 1 \\ 0 \end{array}\right)
 +\frac{\lambda i x }{1+\von f'(0)}\left( \begin{array}{c}
 -1 \\ f'(0) \end{array}\right) + \eta_5(x), \label{assunif} \ee 
where $\|\eta_5(x) \| \leq C(x^2+|x^{-1}f(x)-f'(0)-xf''(0)/2|) $ for all $x$ in some interval 
$[0,x_0]$, $x_0>0$, and the constant $C$ depends only on $\lambda$, $\varepsilon$ and { $\sup_{0<x<x_0}|r(x)|$}.
\end{lemma}
As an immediate corollary of this Lemma, and in particular of the form of the error term, 
if one changes $f$ in a neighbourhood $[0,\delta]$ to some function $f_\delta$, then the 
same asymptotics will hold in $[0,\delta]$ with $f_\delta'(0)$ replacing
$f'(0)$ on the right hand side of (\ref{assunif}). A similar reasoning
works near $x=\pi$ and gives a rigorous justification of the asymptotic
results used in the proof of Theorem \ref{theorem:final}.

\section{Appendix: Linear operator realisation of $L$}\label{section:dom}
The quadratic form associated to $L$ with domain $C^2(-\pi,\pi)$ is strongly indefinite. 
In fact, its numerical range is the whole of $\C$. In this appendix we show that, 
despite of this strong indefinitness, $L$ is a closable 
operator on $C^2(-\pi,\pi)$. The key to this result depends on an analysis of the behaviour
of $u\in \mathcal{D}_{\text{m}}$ in neighbourhoods of $0$, $\pm\pi$.

\begin{lemma}\label{lemma:dom0}
Let $0<\varepsilon<\pi/2$. There exists a constant
$c_\varepsilon>0$, such that
\be
       \int_{-\pi}^\pi |u'(x)|^2\, dx \leq c_\varepsilon \int_{-\pi}^\pi |Lu(x)|^2\, dx
\label{eq:dom1}\ee 
for all $u\in \mathcal{D}_{\text{m}}$. 
\end{lemma}
\begin{proof}
Suppose that 
\be
      -iLu=\varepsilon (fu')'+u'=v \in L^2(-\pi,\pi).
\label{eq:dom2}\ee
 We will only show that
\[
       \int_0^\pi |u'(x)|^2\, dx \leq c_\varepsilon \int_0^\pi |v(x)|^2\, dx.
\] 
If we restrict our attention to $x\in[0,\pi]$, the equation \eqref{eq:dom2}
can be written as
\[
     (pu')'=\frac{p}{\varepsilon f} v,  
\]
where $p$ is given by  \eqref{eq:pdef} and satisfies the asymptotes \eqref{eq:pass}.
Note that $p(x)$ as well as $f(x)$ are continuous and positive in $(0,\pi)$.  
By virtue of Hardy's inequality \eqref{eq:Hardy} below, we can find a constant 
$a_\varepsilon>0$ such that
\begin{align*}
    \int_0^{\pi/2} |u'(x)|^2\, dx & \leq a_\varepsilon 
   \int_0^{\pi/2} \frac{|p(x)u'(x)|^2}{x^{2+\pi/\varepsilon}}\, dx \\
& \leq  a_\varepsilon \int_0^{\pi/2} \frac{|(p(x)u'(x))'|^2}{x^{\pi/\varepsilon}}\, dx \\
&= \frac{a_\varepsilon}{\varepsilon^2}  \int_0^{\pi/2} \frac{|p(x)/f(x)|^2|v(x)|^2}{x^{\pi/\varepsilon}}\, dx \leq
 \frac{a_\varepsilon}{\varepsilon^2} \int_0^{\pi/2} |v(x)|^2 \, dx.
\end{align*}
Similarly, if we change variables $y=\pi-x$, we can find $b_\varepsilon>0$ such that 
\begin{align*}
    \int_0^{\pi/2} |u'(y)|^2\, dy & \leq b_\varepsilon 
   \int_0^{\pi/2} \frac{|p(y)u'(y)|^2}{y^{2-\pi/\varepsilon}}\, dy
\leq \frac{b_\varepsilon}{\varepsilon^2} \int_{0}^{\pi/2} |v(y)|^2 \, dy.
\end{align*}
This ensures the desired inequality.
\end{proof}

In the proof of this lemma observe that Hardy's inequality
\be\label{eq:Hardy}
      \frac{(1-a)^2}{4} \int_0^b t^{a-2} |w(t)|^2\,dt \leq  \int_0^b t^a |w'(t)|^2\, dt 
\ee
holds for any $a\in\R$ if $w$ is absolutely continuous and $\supp(w)\subset(0,b)$.
See \cite[Section~5.3]{Davies_2}.

The fact that $L$ is closed in the maximal domain $\mathcal{D}_{\text{m}}$ 
is a straightforward consequence of this lemma.

\begin{corollary}  \label{lemma:dom1}
   Let $L$ denote the operator defined by taking the closure of the
differential expression at the left side of \eqref{eq:1} in $C^2_{\text{per}}(-\pi,\pi)$.
If the resolvent set of $L$ is non-empty, then the spectrum of $L$ is the set of
eigenvalues of the periodic problem associated with \eqref{eq:1}.  
\end{corollary}
\begin{proof}
Suppose that the resolvent set of $L$ is non-empty, and let $\lambda$ be a member of this
set. Then, by Lemma~\ref{lemma:dom0}, 
$(L-\lambda)^{-1}$ is a bounded operator mapping $L^2(-\pi,\pi)$ into
$H^1(-\pi,\pi)$. As $H^1(-\pi,\pi)$ is compactly contained in $L^2(-\pi,\pi)$,
then necessarily $(L-\lambda)^{-1}$ is a compact operator. The spectrum of $L$ therefore
consists of isolated eigenvalues of finite multiplicity. These are precisely
the set of eigenvalues of the periodic problem associated with \eqref{eq:1}.
\end{proof}

Note that if $f(x)=(2/\pi)\sin(x)$ the hypotheses of this corollary are satisfied.
See \cite{Davies} and \cite{Weir}.


\end{document}